\documentclass[11pt]{article}

\input diagram.tex
\usepackage{amsmath} 
\usepackage[mathscr]{eucal}
\usepackage{amssymb} 
\usepackage{theorem}
\usepackage{tikz-cd}
\usetikzlibrary{positioning}

\usepackage{enumerate}

\newtheorem{theorem}{Theorem}
\newtheorem{corollary}[theorem]{Corollary}
\newtheorem{lemma}[theorem]{Lemma}
\newenvironment{proof}{\noindent{\bf Proof}\ }{\qed\bigskip}

\renewcommand{\ge}{\geqslant}

\newcommand{\FF}{\mathbb{F}}

\newcommand{\qed}{\nobreak\hfill
                  \vbox{\hrule\hbox{\vrule\hbox to 5pt
                  {\vbox to 8pt{\vfil}\hfil}\vrule}\hrule}}

\title{A note on blocks of finite groups with TI Sylow $p$-subgroups}
\author{Deniz Y\i lmaz}
\date{}
\providecommand{\keywords}[1]
{
  \small\smallskip\par	
  \hspace{2ex}\textbf{Keywords:} #1
}
\providecommand{\msc}[1]
{
  \small\smallskip\par	
  \hspace{2ex}\textbf{MSC2020:} #1
}

\begin{document}
\sloppy
\maketitle
\begin{abstract}
Let $\FF$ be an algebraically closed field of characteristic zero. Recently, we proved that isotypic blocks are functorially equivalent over $\FF$. In this article we provide an example of functorially equivalent blocks which are not perfectly isometric.
\end{abstract}

\msc{20C20, 20C34, 20J15.}

\medskip
Let $\FF$ denote an algebraically closed field of characteristic $0$ and let $k$ denote an algebraically closed field of characteristic $p>0$. In modular representation theory, there are different notions of equivalences between blocks of finite groups such as Puig equivalence, splendid Rickard equivalence, $p$-permutation equivalence, isotypies and perfect isometries (\cite{Broue1990}, \cite{BoltjeXu2008}, \cite{BoltjePerepelitsky2020}).  Each equivalence in that list implies the subsequent one and they are related to prominent outstanding conjectures in modular representation theory, such as Brou{\'e}'s abelian defect group conjecture (Conjecture 9.7.6 in~\cite{linckelmann2018}), Puig's finiteness conjecture (Conjecture 6.4.2 in~\cite{linckelmann2018}) and Donovan's conjecture (Conjecture 6.1.9 in~\cite{linckelmann2018}).

In \cite{BoucYilmaz2022} together with Bouc, we introduced another equivalence of blocks, namely functorial equivalences over $R$ where $R$ is a commutative ring. To each pair $(G,b)$ of a finite group $G$ and a block idempotent $b$ of $kG$, we associate a canonical diagonal $p$-permutation functor over $R$. If $(H,c)$ is another such pair, we say that $(G,b)$ and $(H,c)$ are {\em functorially equivalent over $R$} if their associated functors are isomorphic, see \cite[Section~10]{BoucYilmaz2022}.

We proved that the number of isomorphism classes of simple modules, the number of ordinary characters, and the defect groups are preserved under functorial equivalences over $\FF$ (\cite[Theorem~10.5]{BoucYilmaz2022}).  Moreover we proved that for a given finite $p$-group $D$, there are only finitely many pairs $(G,b)$, where $G$ is a finite group and $b$ is a block idempotent of $kG$ with defect groups isomorphic to $D$, up to functorial equivalence over $\FF$ (\cite[Theorem~10.6]{BoucYilmaz2022}) and we provided a sufficient condition for two blocks to be functorially equivalent over $\FF$ in the situation of Brou{\'e}'s abelian defect group conjecture (\cite[Theorem~11.1]{BoucYilmaz2022}).

With Bouc, we also showed that if two blocks are $p$-permutation equivalent, then they are functorially equivalent over $R$, for any $R$ (\cite[Lemma~10.2(ii)]{BoucYilmaz2022}). The natural question hence is to understand the relation between functorial equivalences and isotypies and perfect isometries. Recently, we proved that isotypic blocks are functorially equivalent over $\FF$ and provided an example of perfectly isometric blocks which are not functorially equivalent over $\FF$ \cite{Yilmaz2023}.

The remaining open question is to understand whether a functorial equivalence over $\FF$ implies a perfect isometry. In this paper, we give a negative answer to this question by showing that the principal $2$-block of a simple Suzuki group is functorially equivalent to its Brauer correspondent. 

\begin{theorem}\label{thm mainthm}
Let $G$ be a finite group with TI Sylow $p$-subgroup $P$. Let $b$ be a block idempotent of $kG$ with a defect group $P$ and let $c$ be the block idempotent of $kN_G(P)$ which is in Brauer correspondence with $b$. Then the pairs $(G,b)$ and $(N_G(P),c)$ are functorially equivalent over $\FF$.
\end{theorem}
\begin{proof}
Since $P$ is a TI Sylow $p$-subgroup, by a well-known result (see \cite[Theorem~9.8.6]{linckelmann2018} for instance), the $(kGb,kN_G(P)c)$-bimodule $bkGc$ and its dual induce a stable equivalence of Morita type between $kGb$ and $kN_G(P)c$. In particular, we have a stable $p$-permutation equivalence and hence a stable functorial equivalence over $\FF$ between $(G,b)$ and $(N_G(P),c)$. Moreover, by \cite[Theorem~9.2]{BlauMichler1990}, $kGb$ and $kN_G(P)c$ have the same number of isomorphism classes of simple modules. Therefore, by \cite{BoucYilmaz2023}, there is a functorial equivalence over $\FF$ between $(G,b)$ and $(N_G(P),c)$.
\end{proof}

\begin{corollary}
For $p=2$, the principal $2$-block of the Suzuki group $\mathrm{Sz}(2^{2n+1})$, $n\ge 1$, is functorially equivalent over $\FF$ to its Brauer correspondent. In particular, a functorial equivalence over $\FF$ between blocks does not necessarily imply a perfect isometry.
\end{corollary}
\begin{proof}
The first assertion follows from Theorem~\ref{thm mainthm} and the second assertion follows from the fact the principal $2$-block of $\mathrm{Sz}(2^{2n+1})$ is not perfectly isometric to its Brauer correspondent, see \cite{Cliff2000} or \cite{Robinson2000}.
\end{proof}

\begin{flushleft}
Deniz Y\i lmaz, Department of Mathematics, Bilkent University, 06800 Ankara, Turkey.\\
{\tt d.yilmaz@bilkent.edu.tr}
\end{flushleft}

\end{document}